\documentclass[11pt,letterpaper]{amsart}
\usepackage{fancyhdr}
\usepackage{bm}
\usepackage{hyperref}
\usepackage{graphicx} 
\usepackage{nicematrix}
\usepackage{amsthm,amssymb,amsmath}
\usepackage[T1]{fontenc}
\usepackage[utf8]{inputenc}
\usepackage{hyperref}
\usepackage{lipsum}
\usepackage{mathrsfs}
\usepackage{enumitem}
\usepackage{stmaryrd}
\usepackage{setspace}
\usepackage{yfonts}
\usepackage{float}
\usepackage{mathtools}
\usepackage{parskip}
\usepackage[all,cmtip]{xy}
\usepackage{tikz-cd}
\tikzcdset{row sep/normal=50pt, column sep/normal=50pt}

\newtheorem{lemma}{Lemma}[section]
\newtheorem{remark}[lemma]{Remark}

\newtheorem{theorem*}{Theorem}

\newtheorem{example*}[lemma]{Example}
\newtheorem{proposition}[lemma]{Proposition}
\newtheorem{corollary}[lemma]{Corollary}
\newtheorem{question}[lemma]{Question}
\newtheorem{claim}[lemma]{Claim}

\newtheorem{manualtheoreminner}{Theorem}
\newenvironment{manualtheorem}[1]{%
  \renewcommand\themanualtheoreminner{#1}%
  \manualtheoreminner
}{\endmanualtheoreminner}
\usepackage{blindtext}
\usepackage[margin=1.1in]{geometry}
\usepackage[backend=biber, style=numeric, sorting=nyt]{biblatex}
\title{Stable and unstable syzygy bundles on certain Picard rank one varieties}
\author{\small{Supravat Sarkar}}
\date{}
\begin{document}

\begin{abstract}
   If $X$ is a smooth projective variety of dimension $\geq 2$ and Picard rank $1$ on which every ample line bundle has at least $2$ linearly independent sections, we prove that syzygy bundle of any nontrivial globally generated line bundle is stable. We apply this to show stability of syzygy bundles in non-general type complete intersections in weighted projective spaces, and all complete intersections in projective spaces. This extends earlier results of Jiang-Ren, Coand{\u{a}} and others. We also show that in any dimension $\geq 2$, there is a smooth projective variety of Picard rank $1$ with an unstable syzygy bundle. This completely answers a question of Fulger-Langer.
\end{abstract}
\maketitle
\begin{center}
\textbf{Keywords}: syzygy bundle, stability, complete intersection 
\end{center}
\begin{center}
\textbf{MSC Number: 14H60, 14C20} 
\end{center}

\section{Introduction}
We work throughout over an algebraically closed field $k$ of characteristic $0$. Given a nontrivial globally generated line bundle $L$ on a smooth projective variety $X$, one defines the syzygy bundle $M_L$ to be the kernel of the evaluation map $H^0(X,L)\otimes \mathcal{O}_X\to L.$ It has been an important area of research to decide whether $M_L$ is stable with respect to a fixed polarization. In \cite{ein2013stability}, it was proved that $M_L$ is stable whenever $L$ is sufficiently ample in some sense. 

The problem is more interesting when $X$ has Picard rank $1$, as then the choice of the polarization does not matter. Consider first the case where $X$ is a curve of genus $g$. In this case many results are known. For $g=0$, clearly $M_L$ is never stable whenever $\deg L\geq 2$, as $M_L$ always splits.  For $g\geq 1$, \cite{ein1992stability} proves that $M_L$ is stable whenever $\deg L\geq 2g+1$. This bound was improved in \cite{camere2008stability}, in terms of Clifford index of $X$. \cite{camere2008stability} also shows that there are examples where $M_L$ is unstable. 

On the other hand, when $\dim X\geq 2$ and $X$ has Picard rank $1$, there is no known example where $M_L$ is not stable. So a natural question is the following:
\begin{question}\label{syzygy bundle}
    Let $X$ be a smooth projective variety of Picard rank $1$ and dimension $\geq 2$, and $L$ a nontrivial globally generated line bundle on $X$. Is the syzygy bundle $M_L$ stable?
\end{question}
It appeared in a slightly different form in \cite[Question 2]{fulger2022positivity}. In \cite{ein2013stability}, Ein, Lazarsfield and Mustopa conjectured an affirmative answer to this question when $L$ is sufficiently ample, and this conjecture was proved by Rekuski in \cite{MR4728292}. Semistability of syzygy bundles on projective space was proved in several places, for example, \cite[Corollary 2.2]{flenner1984restrictions}, \cite[Corollary 6.5]{ballico1994restriction} and \cite[Corollary 7.1]{brenner2008looking}. Affirmative answers to Question \ref{syzygy bundle} in several cases are known by \cite{camere2012stability},\cite{mukherjee2022note}, \cite{caucci2021stability}, \cite{jiang2025stability} and other works.

One of the goals of this paper is to give an affirmative answer to Question \ref{syzygy bundle} for a large class of varieties $X$. Our first result is the following.

\begin{manualtheorem}{A}\label{A}
Let $X$ be a smooth projective variety of dimension $\geq 2$ and Picard rank $1$.
Suppose for all ample line bundle $H$ on $X$ we have $h^0(X,H)\geq 2$.
Then for every nontrivial globally generated line bundle $L$ on $X$, the syzygy bundle $M_L$ is stable.
\end{manualtheorem}

As an application of this Theorem, we give affirmative answers to Question \ref{syzygy bundle} when $X$ is a complete intersection in projective space, or a non-general type complete intersection in a weighted projective space $\mathbb{P}$. For projective space, this was known before only for $X$ Fano or Calabi-Yau, see \cite{jiang2025stability}. 

When $\mathbb{P}=\mathbb{P}(a_0, a_1,\cdots, a_n)$ is a weighted projective space, which we may assume is well-formed, and $X$ is a smooth complete intersection in $\mathbb{P}$ of multidegree $(d_1, d_2,\cdots d_c)$, we will make two standard assumptions for $X$ (see \cite{pizzato2017effective}): $X$ is well-formed, that is, $\operatorname{codim}_X\!\left(X\cap \operatorname{Sing}(\mathbb{P})\right)\ge 2,$ and $X$ is not a linear cone, that is, $d_j\neq a_i$ for all $i$ and $j$. When $\mathbb{P}$ is a projective space, the first one always holds, and without loss of generality we can make the second asssumption in purpose of the following theorem, as we can replace $\mathbb{P}$ by the linear span of $X$.
\begin{manualtheorem}{B}\label{B}
Suppose $X$ is a smooth projective variety of dimension $\geq 2$ and
Picard rank $1$, which is a well-formed complete intersection in a weighted projective
space $\mathbb P$, and is not an intersection with a linear cone.
Then $M_L$ is stable for every nontrivial globally generated line bundle $L$ if one
of the following holds:
\begin{enumerate}
\item[(a)] $\mathbb P$ is a projective space.
\item[(b)] $X$ is not of general type.
\end{enumerate}
\end{manualtheorem}

Finally, we show that there are smooth projective varieties $X$ in every dimension $\geq 2$ with a negative answer to Question \ref{syzygy bundle}. Thus, it completely answers \cite[Question 2(1)]{fulger2022positivity} (see Remark \ref{answer}). This also shows that Theorem \ref{B} does not hold for complete intersections in weighted projective spaces if one drops the assumption $(b)$. 
\begin{manualtheorem}{C}\label{C}
Let $X$ be a smooth projective variety of dimension $n\ge2$, and
$\mathcal O_X(1)$ an ample line bundle on $X$ with $\operatorname{Pic}(X)=\mathbb Z\cdot \mathcal O_X(1).$
Suppose $0<k<l$ are integers such that the following hold:

\begin{enumerate}
\item[(i)] $\mathcal O_X(l)$ is globally generated and $h^0(X,\mathcal O_X(l))
<
1+\frac{l}{k}.$

\item[(ii)] The natural map
\[
H^0(X,\mathcal O_X(k))
\otimes
H^0(X,\mathcal O_X(l))
\xlongrightarrow{\eta}
H^0(X,\mathcal O_X(k+l))
\]
is not injective.
\end{enumerate}

Then the syzygy bundle $M_{\mathcal O_X(l)}$ is unstable.

Furthermore, for every $n\ge2$, there is such a triple
$(X,k,l)$ with $X$ a complete intersection in a weighted projective
space.
\end{manualtheorem}
\section{Notation and conventions}
\begin{itemize}
    \item A \textit{variety} is an integral separated $k$-scheme of finite type.
    \item We use the convention of graded rings and modules as in \cite[Section I.2]{hartshorne2013algebraic}. For a graded $k$-algebra $R$ and a graded $R$-module $M=\bigoplus_{j}M_j$ with each $M_j$ finite dimensional over $k$ and $M_j=0$ for all $j<<0$, we define the Hilbert polynomial
    $$P_M(T):=\sum_j(\dim_k M_j)T^j,$$ an element of the ring of Laurent series $\mathbb{Z}((T))$.
    \item For an ample line bundle $H$ on a smooth projective variety $X$, and a nonzero vector bundle $E$ on $X$, $\mu_H(E)$ denotes the slope of $E$ with respect to the polarization $H$. For definitions and more details about slope and stability, see \cite[Chapter 5]{kobayashi2014differential}.
    \item For a normal projective variety $X$, the
Néron--Severi group of $X$ is denoted by $NS(X).$ 
\item If $\mathcal{O}_X(1)$ is a fixed ample line bundle on a projective variety $X$, and $F$ is a coherent sheaf on $X$, we will denote $F\otimes \mathcal{O}_X(1)^{\otimes k}$ as $F(k)$ for all integers $k$.
\item For a line bundle $L$ on a projective variety $X$, we denote the complete linear system given corresponding to $L$ by $|L|$.
\item For a nonzero element $x$ in a free abelian group $A$ of finite rank, the \textit{divisibility} of $x$ is the largest positive integer $n$ such that $x=ny$ for some $y\in A$.
\end{itemize}
\section{A monotonicity result and proof of Theorem \ref{A}}
The goal of this section is to prove Theorem \ref{A}. A key step is Proposition \ref{monotone}, which is interesting in its own right. We first prove the following standard Lemma.
\begin{lemma}\label{Hilbert}
Let $R=\bigoplus_{j\ge0}R_j$ be a graded $k$-algebra, $M=\bigoplus_{j}M_j$ a graded $R$-module, with each $M_j$ finite dimensional over $k$, and
$s\in R_1$ be a nonzero divisor in $M$. Let $N=M/sM$. Then
\[
P_M(T)=(1-T)^{-1}P_N(T).
\]
\end{lemma}

\begin{proof}
For every $j\in\mathbb{Z}$, we have an exact sequence of $k$-vector spaces
\[
0\longrightarrow M_{j-1}
\xrightarrow{\cdot s}
M_j
\longrightarrow N_j
\longrightarrow0.
\]
Hence
\[
\dim N_j=\dim M_j-\dim M_{j-1}.
\]
Therefore,
\begin{align*}
P_N(T)
&=\sum_j(\dim N_j)\,T^j \\
&=\sum_j(\dim M_j)\,T^j-\sum_j(\dim M_{j-1})\,T^j \\
&=\sum_j(\dim M_j)\,T^j
-T\sum_j(\dim M_{j-1})\,T^{j-1} \\
&=(1-T)\sum_j(\dim M_j)\,T^j 
=(1-T)P_M(T).
\end{align*}
Thus
\[
P_M(T)=(1-T)^{-1}P_N(T).
\]
\end{proof}
\begin{proposition}\label{monotone}
Let $X$ be a normal projective variety of dimension $\geq 2$,
$\mathcal{O}_X(1)$ an ample line bundle on $X$ such that the base locus of $|\mathcal{O}_X(1)|$ has codimension $\geq 2$ in $X$, and $F$ a numerically trivial line bundle on $X$.
Then for all $k\geq 1$ we have
$$\frac{h^0(X,F(k))-1}{k}\leq \frac{h^0(X,F(k+1))-1}{k+1}.$$
Further, if equality holds for some $k$, then $F=\mathcal{O}_X$ and $\mathcal{O}_X(k+1)$ is not globally generated.
\end{proposition}
\begin{proof}
For $j\in \mathbb{Z}$, let $R_j=H^0(X,\mathcal{O}_X(k))$ and $M_j=H^0(X,F(j))$. So, $R_j=M_j=0$ for $j<0$, and $F$ is numerically trivial. Let $R=\bigoplus_{j\ge0}R_j$, $M=\bigoplus_{j\ge0}M_j$. So, $R$ is a graded $k$-algebra and $M$ a graded $R$-module with each $M_j$ finite dimensional over $k$. Let $s,t\in R_1$ be general.

\begin{claim}\label{regular sequence}
 $(s,t)$ is an $M$-regular sequence.
\end{claim}
\begin{proof}
Clearly $s$ is a nonzero divisor in $M$. We need to show the image of $t$ is a nonzero divisor in $M/sM$. Equivalently, if $j\geq 0$ and $
tf=sg$
for some nonzero $f,g\in M_j$, we
need to show that $f\in sM$.

Let $D_s,\ D_t,\ D_f$ and $\ D_g$ denote the
effective Cartier divisors defined by the sections $s,t,f$ and $g$,
respectively. So
\begin{equation}\label{D}
D_t+D_f=D_s+D_g.
\end{equation}

 As $s$ and $t$ are general and the base locus of $|\mathcal{O}_X(1)|$ has codimension $\geq 2$ in $X$, we see that
$D_s$ and $D_t$ have no common irreducible component. So by \eqref{D} we have $D_s\le D_f,$
hence the rational section $\frac{f}{s}$
of $F(j-1)$ has no poles. Since $X$ is normal, $
\frac{f}{s}$ must be a regular section,
so $f\in sM$.
\end{proof}
Let $N =\frac{M}{(s,t)M},$
a graded $R$-module. Write $a_j=\dim N_j,$ for $j\ge0.$
\begin{claim}\label{2}
If $F=\mathcal{O}_X$ and $\mathcal{O}_X(1)$ is globally generated, then we have $a_j\ge 1$ for all $j\ge 0$.
\end{claim}
\begin{proof}
Suppose $F=\mathcal{O}_X$, so $M=R$. As $\dim X\ge2$, and $\mathcal{O}_X(1)$ is ample, there is $p\in X$ where both $s$ and $t$
vanish. If $\mathcal{O}_X(1)$ is globally generated, there is $u\in H^0(X,\mathcal{O}_X(1))$ not vanishing at $p$. Then $u^j\in R_j$, but not in the ideal $(s,t)$. So, $N_j\neq 0$ for all $j\geq 0$.
\end{proof}

 Using Claim \ref{regular sequence} and repeatedly applying Lemma \ref{Hilbert}
we have
\[
P_M(T)=(1-T)^{-2}P_N(T).
\]
As
\[
(1-T)^{-2}
=\sum_{i=0}^{\infty}(i+1)T^i,
\]
we have
\begin{equation}\label{P_R}
P_M(T)
=
\left(\sum_{j=0}^{\infty}a_jT^j\right)
\left(\sum_{i=0}^{\infty}(i+1)T^i\right).
\end{equation}

Equating the coefficients of $T^k$ in both sides of \eqref{P_R}, we get
\[
h^0(X,F(k))
=
\sum_{j=0}^{k}a_j(k-j+1).
\]
So, for $k\ge 1$ we have
\begin{align*}
&\frac{h^0(X,F(k+1))-1}{k+1}
-\frac{h^0(X,F(k))-1}{k} \\
&=
\sum_{j=0}^{k+1}\left(1-\frac{j-1}{k+1}\right)a_j
-
\sum_{j=0}^{k}\left(1-\frac{j-1}{k}\right)a_j +\frac{1}{k(k+1)}\\
&=
\frac{1-a_0}{k(k+1)}
+\frac{1}{k(k+1)}
\sum_{j=1}^{k}(j-1)a_j
+\frac{a_{k+1}}{k+1} \\
&=
\frac{1}{k(k+1)}
\left(
\sum_{j=2}^{k}(j-1)a_j
+(ka_{k+1}+1-a_0)
\right).
\end{align*}
If $F$ is non-trivial, then $a_0=h^0(X,F)=0$ as $F\equiv0$, so the last expression is $> 0$. If $F$ is trivial, then $a_0=1$ and the last expression is $\geq 0$, and it is $>0$ if further $\mathcal{O}_X(1)$ is globally generated,  by Claim \ref{2}.
\end{proof}
\begin{corollary}\label{key}
Let $H$ be an ample line bundle on a normal projective variety $X$ of
dimension $\geq 2$, and assume that for every line bundle $H_1\equiv H$, the base locus of $|H_1|$ has codimension $\geq 2$ in $X$. If $k>0$ is an integer and
$N_1,N_2\equiv0$ are line bundles on $X$, then
\begin{equation}\label{key1}
\frac{h^0(X,N_1\otimes H^k)-1}{k}
\le
\frac{h^0(X,N_2\otimes H^{k+1})-1}{k+1},
\end{equation} and the inequality is strict if $N_2\otimes H^{k+1}$ is globally generated.
\end{corollary}

\begin{proof}
Let
\[
H_1=H\otimes N_2\otimes N_1^{-1},
\qquad
F=N_1^{k+1}\otimes N_2^{-k}.
\]
So, $H_1\equiv H$ is ample and the base locus of $|H_1|$ has codimension $\geq 2$ in $X$ by assumption. Also $F\equiv0.$
Observe that
\[
F\otimes H_1^k
\cong
N_1\otimes H^k,
\qquad
F\otimes H_1^{k+1}
\cong
N_2\otimes H^{k+1}.
\]
Now we are done by Lemma \ref{monotone}.
\end{proof}
We will need the following Lemma in the proof of Theorem \ref{A}, which was essentially proved in \cite{jiang2025stability}.
\begin{lemma}\label{jr main}
Let $X$ be a smooth projective variety of dimension $\geq 2$ and Picard rank $1$, and $\mathcal O_X(1)$ is an ample line bundle on $X$ whose class generates
$NS(X)/\mathrm{torsion}\cong\mathbb Z$. Let $N_2\equiv0$ be a line bundle on $X$, $l>0$ an
integer, such that $L:=N_2(l)$
is globally generated.
Suppose further that for all integer $k$ with $0<k<l$ and all line bundles
$N_1\equiv0$, we have
\begin{equation}\label{jr}
\frac{h^0(X,N_1(k))-1}{k}
<
\frac{h^0(X,N_2(l))-1}{l}.
\end{equation}

Then $M_L$ is stable.
\end{lemma}
\begin{proof}
    The same proof as \cite[Theorem 1.1]{jiang2025stability} works.
\end{proof}
Now we are ready to prove Theorem \ref{A}.

\textit{Proof of Theorem \ref{A}:} 
Let $H=\mathcal O_X(1)$ be an ample line bundle on $X$ whose class generates
$NS(X)/\mathrm{torsion}\cong\mathbb Z$. Recall that the kernel of the natural surjection $\operatorname{Pic}(X)\longrightarrow NS(X)/\mathrm{torsion}$
is the group of numerically trivial line bundles (up to isomorphism).
So, there is $l\ge1$ and a line bundle
$N_2\equiv0$ such that $L\cong N_2(l).$

If $H_1\equiv H$ is any line bundle, then the class of $H_1$ also generates
$NS(X)/\mathrm{torsion}$, hence $H_1$ is not linearly equivalent to sum of two nonzero effective divisors. This shows that any element in $|H_1|$ is irreducible. Since $H_1$ is ample, $|H_1|$ is not singleton by assumption. Hence the base locus of $|H_1|$ has codimension $\geq 2$ in $X$. 

Hence by repeated application of Lemma \ref{key}, the inequality \eqref{jr} is satisfied for all $0<k<l$. Now Theorem \ref{A} follows from Lemma \ref{jr main}.
\section{Stability of syzygy bundles on complete intersections}
The goal of this section is to prove Theorem \ref{B}.

\textit{Proof of Theorem \ref{B}:}
If $X$ is a surface not of general type, then we are done by
\cite[Theorem 1.4]{jiang2025stability}.
So assume either $\dim X\ge3$ or $X$ is a surface of general type.
By Theorem \ref{A}, it suffices to show that:
\[
\operatorname{Pic}(X)=\mathbb Z\cdot\mathcal O_X(1)
\quad\text{and}\quad
h^0(X,\mathcal O_X(1))\ge2.
\]

Suppose \emph{(a)} holds. Then $h^0(X,\mathcal O_X(1))\ge2,$
as $\mathcal O_X(1)$ is very ample and $\dim X\ge2$. If $\dim X\ge3$, then $\operatorname{Pic}(X)=\mathbb Z\cdot\mathcal O_X(1)$
by the Grothendieck--Lefschetz theorem. If $\dim X=2$, then $X=Y\cap H$ in $\mathbb P,$
where $Y$ is a smooth $3$-dimensional complete intersection and $H$ is a
hypersurface in $\mathbb P$. By the Grothendieck--Lefschetz theorem (see \cite[Section 3.1]{lazarsfeld2017positivity}), $\operatorname{Pic}(Y)=\mathbb Z\cdot\mathcal O_Y(1).$
Note that $\mathcal O_Y(K_Y+X)\cong\mathcal O_Y(m)$
for some integer $m$, and by adjunction, $\omega_X\cong \mathcal O_X(m).$
As $X$ is of general type, we have $m>0$, so
$\mathcal O_Y(K_Y+X)$ is very ample. By
\cite[Theorem 1]{ravindra2009noether}, a very general element
$X'$ of $|\mathcal O_Y(X)|$ has $\operatorname{Pic}(X')=\mathbb Z\cdot\mathcal O_{X'}(1).$

The smooth members $X'$ of $|\mathcal O_Y(X)|$ form a smooth proper family, and
the line bundles $\mathcal O_{X'}(1)$ give a line bundle on the family.
The family is topologically locally trivial by Ehresmann's fibration theorem (\cite{ehresmann1950connexions}),
and by Lefschetz hyperplane theorem (\cite[Theorem 3.1.21]{lazarsfeld2017positivity}) all $X'$ in the family are simply connected.
Hence $\operatorname{Pic}(X')$ is free abelian. Therefore, the divisibility of
$\mathcal O_{X'}(1)$ in $\operatorname{Pic}(X')$ is constant in the family.
As for very general $X'$ this divisibility is $1$, we get $\operatorname{Pic}(X')=\mathbb Z\cdot\mathcal O_{X'}(1)$
for every member of the family, in particular for $X'=X$.

Now suppose \emph{(b)} holds. So $\dim X\ge3$ by our assumption in the
beginning of the proof. Let $\mathbb P=\mathbb P(a_0,a_1,\ldots,a_n),$
which we may assume is well formed.

By \cite[Proposition 2.3]{pizzato2017effective}, $\operatorname{Pic}(X)=\mathbb Z\cdot\mathcal O_X(1).$ Since $X$ is not a Calabi--Yau surface,
\cite[Corollary 5.3(i)]{pizzato2017effective}
shows that at least two of the weights $a_i$ are equal to $1$.
Therefore, $h^0(\mathbb P,\mathcal O_{\mathbb P}(1))\ge2.$ As $X$ is not an intersection with a linear cone, the homogeneous ideal of $X$
contains no homogeneous polynomial of degree $1$. Hence the natural restriction
map
\[
H^0(\mathbb P,\mathcal O_{\mathbb P}(1))
\longrightarrow
H^0(X,\mathcal O_X(1))
\]
is injective. So, $h^0(X,\mathcal O_X(1))\ge2.$
\section{Existence of unstable syzygy bundles}
The goal of this section is to prove Theorem \ref{C}. We will need the following Lemma.

\begin{lemma}\label{iso}
Let $\mathbb{P}$ be a weighted projective space of dimension $N\ge3$, and fix
integers $2\le n\le N$ and $t\ge0$. Then for all sufficiently divisible positive integers
$d_1,\ldots,d_{N-n}$, a general complete intersection
$X$ in $\mathbb{P}$ of multidegree
$(d_1,\ldots,d_{N-n})$ satisfies the following:

\begin{enumerate}
\item[(i)]
The natural map
\[
H^0(\mathbb{P},\mathcal{O}_{\mathbb{P}}(j))
\longrightarrow
H^0(X,\mathcal{O}_X(j))
\]
is an isomorphism for all $j\le t$.

\item[(ii)] $H^i(X,\mathcal{O}_X(j))=0$
for all $j\in\mathbb{Z}$ and $0<i<n$.
\end{enumerate}
\end{lemma}

\begin{proof}
Fixing $N$, we do reverse induction on $n$. For $n=N$, we have $X=\mathbb{P}$, so $(i)$ is
clear and $(ii)$ follows from \cite[Section 1.4]{dolgachev2006weighted}. In general, write $X=Y\cap H,$
where $Y$ is a general complete intersection of multidegree
$(d_1,\ldots,d_{N-n-1})$, and $H$ is a general hypersurface in
$\mathbb{P}$ of degree $d_{N-n}$. As $H$ is general, it does not contain any associated point of the coherent
sheaf $\mathcal{O}_Y(j)$ for all $j\in\mathbb{Z}$. So, the local equation of
the effective Cartier divisor $X$ in $Y$ is a nonzero divisor in every stalk
of $\mathcal{O}_Y(j)$. Hence we have a short exact sequence
\begin{equation}\label{restriction seq}
0
\longrightarrow
\mathcal{O}_Y(j-d_{N-n})
\longrightarrow
\mathcal{O}_Y(j)
\longrightarrow
\mathcal{O}_X(j)
\longrightarrow
0.
\end{equation}

Using the long exact sequence in cohomology coming from \eqref{restriction seq} and induction, we
get $(ii)$, and also an exact sequence
\begin{equation}\label{h0}
0
\longrightarrow
H^0\!\left(Y,\mathcal{O}_Y(j-d_{N-n})\right)
\longrightarrow
H^0\!\left(Y,\mathcal{O}_Y(j)\right)
\longrightarrow
H^0\!\left(X,\mathcal{O}_X(j)\right)
\longrightarrow
0.
\end{equation}

Now, as $d_{N-n}$ is sufficiently divisible, we have $d_{N-n}>t\ge j,$
so $j-d_{N-n}<0.$
By induction,
\[
H^0\!\left(Y,\mathcal{O}_Y(j-d_{N-n})\right)
\cong
H^0\!\left(\mathbb{P},\mathcal{O}_{\mathbb{P}}(j-d_{N-n})\right)
=0,
\]
So, by \eqref{h0} and induction we get $(i)$.
\end{proof}
\textit{Proof of Theorem \ref{C}:}
Note that
\begin{align*}
H^0(X,\mathcal O_X(k))\otimes H^0(X,\mathcal O_X(l))
&=
H^0\bigl(X,H^0(X,\mathcal O_X(l))\otimes_{\mathbb C}\mathcal O_X(k)\bigr)\\
&=
H^0\Bigl(X,
\mathcal Hom_{\mathcal O_X}
\bigl(\mathcal O_X(-k),
H^0(X,\mathcal O_X(l))\otimes_{\mathbb C}\mathcal O_X
\bigr)
\Bigr)\\
&=
\operatorname{Hom}_{\mathcal O_X}
\bigl(
\mathcal O_X(-k),
H^0(X,\mathcal O_X(l))\otimes_{\mathbb C}\mathcal O_X
\bigr),
\end{align*}
and
\begin{align*}
H^0(X,\mathcal O_X(k+l))
&=
H^0\bigl(X,
\mathcal Hom_{\mathcal O_X}
(\mathcal O_X(-k),\mathcal O_X(l))
\bigr)\\
&=
\operatorname{Hom}_{\mathcal O_X}
(\mathcal O_X(-k),\mathcal O_X(l)).
\end{align*}

Via these identifications, $\eta$ is the map
induced by the evaluation map
\[
H^0(X,\mathcal O_X(l))
\otimes_{\mathbb C}\mathcal O_X
\longrightarrow
\mathcal O_X(l).
\]

Since the functor $Hom_{\mathcal{O}_X}\bigl(\mathcal{O}_X(-k),\,\underline{\phantom{X}}\bigr)$ is left exact, we have $$
\ker\eta
= Hom_{\mathcal O_X}
(\mathcal O_X(-k),M_{\mathcal O_X(l)}).
$$

As $\ker\eta\neq0$, we see that $\mathcal O_X(-k)$ is a subsheaf of
$M_{\mathcal O_X(l)}$.

Now
\[
\mu_H(\mathcal O_X(-k))
=
-k\,H^n,
\]
and
\[
\mu_H(M_{\mathcal O_X(l)})
=
-\frac{l}{h^0(X,\mathcal O_X(l))-1}\,H^n.
\]

Now by $(i)$
we see that $\mathcal O_X(-k)$ is a destabilizing subsheaf of
$M_{\mathcal O_X(l)}$. This proves the first part of the theorem.

Now we prove the second part. Fix $n\ge2$. Choose a sufficiently large integer $m$. Let
\[
a=4(m+4),\qquad
b=a+1,
\]
and choose a prime $p>2m$. Let $S=\mathbb C[x,y,z,w,u_1,v_1,\ldots,u_m,v_m]$, and $\mathbb P=\operatorname{Proj} S$, a weighted projective space
of dimension $2m+3$, where the degrees of the variables are given by
\begin{align*}
\deg x
&=pa
=p(4m+16),\\
\deg y
&=pb
=p(4m+17),\\
\deg z
&=p(a+b)
=p(8m+33),\\
\deg w
&=p(ab-a-b)
=p(16m^2+124m+239),\\
\deg u_i
&=p(2m+8)(4m+17)-i,\\
\deg v_i
&=p(2m+8)(4m+17)+i.
\end{align*}
It is easy to see that $\mathbb P$ is well-formed, this will also follow from the proof of the following Claim.
\begin{claim}\label{sing}
$\dim\operatorname{Sing}\mathbb P\le m$.
\end{claim}

\begin{proof}
Suppose not. We want to get a contradiction.
The irreducible components of $\operatorname{Sing}\mathbb P$ are obtained by
choosing a subset $T'$ of $T:=\{x,y,z,w,u_1,v_1,\ldots,u_m,v_m\}$,
the degrees of whose elements have gcd $>1$, and setting the
variables in $T\setminus T'$ equal to $0$. Thus, each such irreducible
component has dimension $|T'|-1$. Hence there is one such $T'$ with $|T'|\ge m+2.$
Let $q$ be a prime dividing the degrees of every element of $T'$.

As $p>2m$, we have
\[
p\nmid\deg u_i,\qquad
p\nmid\deg v_i,
\]
for all $1\le i\le m$. Since $m+2>4$, we get $q\neq p$.

As $a$, $b$, $a+b$, and $ab-a-b$ are pairwise coprime, $q$ can divide at
most one of the degrees of $x,y,z,w$. Therefore,
\[
\left|T'\cap
\{u_1,v_1,\ldots,u_m,v_m\}\right|
\ge m+1.
\]

But
\[
\{\deg u_1,\deg v_1,\ldots,\deg u_m,\deg v_m\}
\]
is a set of $2m$ consecutive integers. Hence degrees of some two elements in $
T'\cap\{u_1,v_1,\ldots,u_m,v_m\}$
must differ by $1$. But they are both divisible by $q$, a contradiction.
\end{proof}

Let
\[
l=pab
=p(4m+16)(4m+17),
\]
and
\[
k=p(a+b)=p(8m+33).
\]

\begin{claim}\label{bs}
\begin{enumerate}
\item[(1)] $\dim \operatorname{Bs}|\mathcal O_{\mathbb P}(l)|\le m$.

\item[(2)] $
h^0(\mathbb P,\mathcal O_{\mathbb P}(l))
<
1+\frac{l}{k}.$
\end{enumerate}
\end{claim}

\begin{proof}
Both will follow once we show
\begin{equation}\label{sl}
S_l=
\operatorname{span}
\left\{
x^b,\,
y^a,\,
zw,\,
xyw,\,
u_1v_1,\,
u_2v_2,\,
\ldots,\,
u_mv_m
\right\}.
\end{equation}

Let $x^\alpha y^\beta z^\gamma w^\delta
\prod_i u_i^{s_i}v_i^{t_i}$
be a monomial of degree $l$. So,
\begin{equation}\label{deg}
\alpha\deg x+\beta\deg y+\gamma\deg z+\delta\deg w
+\sum_{i=1}^{m}\left(s_i\deg u_i+t_i\deg v_i\right)=l.
\end{equation}

Note that
\[
\deg u_i=\frac{l}{2}-i\ge \frac{l}{2}-m>\frac{l}{3},
\]
and similarly $\deg v_i>\frac{l}{3}$. So, using \eqref{deg},
\begin{align*}
l
&\ge
\sum_i\left(s_i\deg u_i+t_i\deg v_i\right)
>
\frac{l}{3}\sum_i(s_i+t_i).
\end{align*}
So,
\begin{equation}\label{s+t}
\sum_i(s_i+t_i)\le 2.
\end{equation}

Also, reducing \ref{deg} modulo $p$, we get
\begin{equation}\label{mod}
\sum_{i=1}^m i(t_i-s_i)\equiv 0 \pmod p.
\end{equation}

Suppose not all $s_i,t_i$ are zero. As $p>2m$, from \eqref{s+t} and \eqref{mod} we see the only possibility is that
$s_{i_0}=t_{i_0}=1$ for some $i_0$, and
$s_i=t_i=0$ for all $i\neq i_0$. So, $M=u_iv_i.$

Now suppose $s_i=t_i=0$ for all $i$. Note that
$\deg w>\frac{l}{2}$ as $m$ is large. So, by \eqref{deg},
\[
l\ge \delta\deg w>\frac{\delta l}{2}.
\]
Hence $\delta<2$, that is, $\delta=0$ or $1$.

If $\delta=1$, \eqref{deg} gives
\[
\alpha(4m+16)+\beta(4m+17)+\gamma(8m+33)=8m+33,
\]
forcing $(\alpha,\beta,\gamma)=(1,1,0)
\text{ or }
(0,0,1).$
Hence
\[
M=xyw \quad\text{or}\quad zw.
\]

If $\delta=0$, \eqref{deg} gives
\[
(\alpha+\gamma)a+(\beta+\gamma)b=ab.
\]
As $\gcd(a,b)=1$, we get
\[
a\mid (\beta+\gamma)
\quad\text{and}\quad
b\mid (\alpha+\gamma).
\]
This forces $(\alpha+\gamma,\beta+\gamma)=(b,0)
\text{ or }
(0,a).$
Hence $\gamma=0$ and $(\alpha,\beta)=(b,0)
\text{ or }
(0,a),$
so
\[
M=x^b \quad\text{or}\quad y^a.
\]

This proves \eqref{sl}, hence the Claim follows.
\end{proof}
Now choose positive integers $d_1,\ldots,d_{2m+3-n}$ sufficiently divisible, and let
$X$ be a very general complete intersection in $\mathbb{P}$ of multidegree
$(d_1,\ldots,d_{2m+3-n})$. As $m\gg0$, we have $2m+3-n>m$. So by Claim \ref{sing}, $X\cap \operatorname{Sing}\mathbb{P}=\varnothing,$
and hence $X$ is smooth by Bertini's theorem. Each $\mathcal{O}_X(i)$ is a line
bundle on $X$, and by repeated application of
\cite[Theorem 1]{ravindra2005grothendieck}, and
\cite[Theorem 1]{ravindra2009noether}, we get $\operatorname{Pic}X=\mathbb{Z}\cdot\mathcal{O}_X(1).$
Also, by Claim \ref{bs}, $X\cap \operatorname{Bs}|\mathcal{O}_{\mathbb{P}}(l)|=\varnothing,$
so $\mathcal{O}_X(l)$ is globally generated.

By Lemma \ref{iso}, we have
\[
H^0(X,\mathcal{O}_X(i))
\cong
H^0(\mathbb{P},\mathcal{O}_{\mathbb{P}}(i))
\cong
S_i,
\]
for all $i\le k+l$. So by Claim~2,
\[
h^0(X,\mathcal{O}_X(l))
<
1+\frac{l}{k}.
\]
Hence $(i)$ of Theorem \ref{C} holds. To show that $(ii)$ of Theorem \ref{C} also holds, it suffices to show that the
multiplication map $\eta_1:S_k\otimes S_l\longrightarrow S_{k+l}$
is not injective. Note that
\[
xy,\ z\in S_k,
\qquad\text{and}\qquad
zw,\ xyw\in S_l.
\]
So
\[
Q:=xy\otimes zw-z\otimes xyw
\]
is a nonzero element of $S_k\otimes S_l$, and $\eta_1(Q)=0.$
Hence $(ii)$ of Theorem \ref{C} also holds.

\begin{remark}\label{answer}
In this remark we follow the convention of the projectivization of a vector bundle as in \cite[Chapter 2, Section 7]{hartshorne2013algebraic}. For a vector bundle $E$, its dual is denoted by $E^*$.

If $L$ is a globally generated ample line bundle on a smooth projective variety
$X$ of Picard rank $1$ and dimension $n\ge2$, then $Y:=\mathbb{P}_X(M_L^{*})$
has a nonconstant map to $\mathbb{P}\!\left(H^0(X,M_L^{*})\right),$
given by the complete linear system $|\mathcal{O}_Y(1)|$. As
\[
\dim Y
=
n+h^0(X,L)-2
>
h^0(X,L)-1
=
\dim \mathbb{P}\!\left(H^0(X,M_L^{*})\right),
\]
the contractions of both rays of the Mori cone of curves in $Y$ are of fibre type.
Hence the nef and pseudoeffective cone of $Y$ are the same, that is, $Y$ is $1$-homogeneous in the sense of
\cite{fulger2022positivity}.

Thus, Theorem \ref{C} answers \cite[Question 2(1)]{fulger2022positivity} affirmatively. The other questions \cite[Question 2(2)]{fulger2022positivity} and \cite[Question 2(3)]{fulger2022positivity} are relatively easier. For \cite[Question 2(2)]{fulger2022positivity}, one can take $X=\mathbb{P}^2$, and 
$\mathcal{E}=\mathcal{O}_{\mathbb{P}^2}\oplus M_{\mathcal{O}_{\mathbb{P}^2}(2)}^*,$
and apply \cite[Theorem 7.1]{bansal2024extremal}, or argue as above. See \cite{liu2026slope} for other
examples. For  \cite[Question 2(3)]{fulger2022positivity}, one can take $
X=\mathbb{P}^5
\text{ and }
\mathcal{E}
=
\mathcal{O}_{\mathbb{P}^5}^{\oplus 3}
\oplus
\Omega_{\mathbb{P}^5}(2),$
and argue as above.
\end{remark}
\begin{remark}
    We were informed that Jihao Liu also independently constructed the following counterexample giving negative answer to Question \ref{syzygy bundle},  using Chatgpt 5.6-Sol-Ultra and the Danus system:
    $X$= a general complete intersection $3$-fold in the weighted projective space $\mathbb{P}(2,2,2,2,9,9,9,9)$ of multidegree $(18,18,18,18)$, and $L=\mathcal{O}_X(11).$ It has a destabilizing subbundle isomorphic to $M_{\mathcal{O}_X(2)}^{\oplus 4}$.
\end{remark}
\section{Acknowledgement}
 I thank János Kollár, Robert Lazarsfield and Justin Lacini for insightful discussions, and Shivam Vats and Ashima Bansal for  making me interested in \cite[Question 2]{fulger2022positivity}.
\printbibliography

@article{fulger2022positivity,
  title={Positivity vs. slope semistability for bundles with vanishing discriminant},
  author={Fulger, Mihai and Langer, Adrian},
  journal={Journal of Algebra},
  volume={609},
  pages={657--687},
  year={2022},
  publisher={Elsevier}
}

@article{jiang2025stability,
  title={Stability of syzygy bundles on varieties of Picard number one},
  author={Jiang, Chen and Ren, Peng},
  journal={International Mathematics Research Notices},
  volume={2025},
  number={8},
  pages={rnaf098},
  year={2025},
  publisher={Oxford University Press}
}

@article{ein2013stability,
  title={Stability of syzygy bundles on an algebraic surface},
  author={Ein, Lawrence and Lazarsfeld, Robert and Mustopa, Yusuf},
  journal={Mathematical Research Letters},
  volume={20},
  number={1},
  pages={73--80},
  year={2013},
  publisher={International Press of Boston, Inc. Somerville, MA 02143, USA}
}

@article{ein1992stability,
  title={Stability and restrictions of Picard bundles, with an application to the normal bundles of elliptic curves},
  author={Ein, Lawrence and Lazarsfeld, Robert},
  journal={Complex projective geometry (Trieste, 1989/Bergen, 1989)},
  volume={179},
  pages={149--156},
  year={1992},
  publisher={Cambridge Univ. Press Cambridge}
}

@article{camere2008stability,
  title={About the stability of the tangent bundle restricted to a curve},
  author={Camere, Chiara},
  journal={Comptes Rendus. Math{\'e}matique},
  volume={346},
  number={7-8},
  pages={421--426},
  year={2008}
}

@article {MR4728292,
    AUTHOR = {Rekuski, Nick},
     TITLE = {Stability of kernel sheaves associated to rank one
              torsion-free sheaves},
   JOURNAL = {Math. Z.},
  FJOURNAL = {Mathematische Zeitschrift},
    VOLUME = {307},
      YEAR = {2024},
    NUMBER = {1},
     PAGES = {Paper No. 2, 18},
      ISSN = {0025-5874,1432-1823},
   MRCLASS = {14D20 (14F06)},
  MRNUMBER = {4728292},
MRREVIEWER = {R.\ Parthasarathi},
       DOI = {10.1007/s00209-024-03475-y},
       URL = {https://doi.org/10.1007/s00209-024-03475-y},
}

@article{camere2012stability,
  title={About the stability of the tangent bundle of restricted to a surface},
  author={Camere, Chiara},
  journal={Mathematische Zeitschrift},
  volume={271},
  number={1},
  pages={499--507},
  year={2012},
  publisher={Springer}
}

@article{mukherjee2022note,
  title={A note on stability of syzygy bundles on Enriques and bielliptic surfaces},
  author={Mukherjee, Jayan and Raychaudhury, Debaditya},
  journal={Proceedings of the American Mathematical Society},
  volume={150},
  number={09},
  pages={3715--3724},
  year={2022}
}

@article{caucci2021stability,
  title={Stability of syzygy bundles on abelian varieties},
  author={Caucci, Federico and Lahoz, Mart{\'\i}},
  journal={Bulletin of the London Mathematical Society},
  volume={53},
  number={4},
  pages={1030--1036},
  year={2021},
  publisher={Wiley Online Library}
}

@book{hartshorne2013algebraic,
  title={Algebraic geometry},
  author={Hartshorne, Robin},
  year={2013},
  publisher={Springer Science \& Business Media}
}

@book{kobayashi2014differential,
  title={Differential geometry of complex vector bundles},
  author={Kobayashi, Shoshichi},
  year={2014},
  publisher={Princeton University Press}
}

@article{bansal2024extremal,
  title={Extremal Contraction of Projective Bundles},
  author={Bansal, Ashima and Sarkar, Supravat and Vats, Shivam},
  journal={arXiv preprint arXiv:2409.05091},
  year={2024}
}

@article{liu2026slope,
  title={A slope-unstable bundle on a surface with 1-homogeneous projectivization},
  author={Liu, Jihao},
  journal={arXiv preprint arXiv:2607.04376},
  year={2026}
}

@article{ravindra2005grothendieck,
  title={The Grothendieck-Lefschetz theorem for normal projective varieties},
  author={Ravindra, Girivaru V and Srinivas, Vasudevan},
  journal={arXiv preprint math/0511134},
  year={2005}
}

@article{ravindra2009noether,
  title={The Noether--Lefschetz theorem for the divisor class group},
  author={Ravindra, Girivaru V and Srinivas, Vasudevan},
  journal={Journal of Algebra},
  volume={322},
  number={9},
  pages={3373--3391},
  year={2009},
  publisher={Elsevier}
}

@article{pizzato2017effective,
  title={Effective nonvanishing for Fano weighted complete intersections},
  author={Pizzato, Marco and Sano, Taro and Tasin, Luca},
  journal={Algebra \& Number Theory},
  volume={11},
  number={10},
  pages={2369--2395},
  year={2017},
  publisher={Mathematical Sciences Publishers}
}

@inproceedings{dolgachev2006weighted,
  title={Weighted projective varieties},
  author={Dolgachev, Igor},
  booktitle={Group Actions and Vector Fields: Proceedings of a Polish-North American Seminar Held at the University of British Columbia January 15--February 15, 1981},
  pages={34--71},
  year={2006},
  organization={Springer}
}

@article{flenner1984restrictions,
  title={Restrictions of semistable bundles on projective varieties},
  author={Flenner, Hubert},
  journal={Commentarii Mathematici Helvetici},
  volume={59},
  number={1},
  pages={635--650},
  year={1984},
  publisher={Springer}
}

@article{ballico1994restriction,
  title={On the restriction of vector bundles to good families of curves},
  author={Ballico, E},
  journal={Annali di Matematica Pura ed Applicata},
  volume={167},
  number={1},
  pages={101--116},
  year={1994},
  publisher={Springer}
}

@article{brenner2008looking,
  title={Looking out for stable syzygy bundles},
  author={Brenner, Holger},
  journal={Advances in Mathematics},
  volume={219},
  number={2},
  pages={401--427},
  year={2008},
  publisher={Elsevier}
}

@article{ehresmann1950connexions,
  title={Les connexions infinit{\'e}simales dans un espace fibr{\'e} diff{\'e}rentiable},
  author={Ehresmann, Charles},
  year={1950},
  publisher={se}
}

@book{lazarsfeld2017positivity,
  title={Positivity in algebraic geometry I: Classical setting: line bundles and linear series},
  author={Lazarsfeld, Robert K},
  volume={48},
  year={2017},
  publisher={Springer}
}
\vspace{40pt}
\begin{flushleft}
{\scshape Department of Mathematics, Fine Hall, Princeton University, Princeton, NJ 700108, USA}.

{\fontfamily{cmtt}\selectfont
\textit{Email address: ss6663@princeton.edu} }
\end{flushleft}
\end{document}